\documentclass{article}
\usepackage{amsmath,amssymb,amsthm,bm}
\usepackage[cp1250]{inputenc}
\usepackage{enumerate}
\usepackage{a4wide}

\numberwithin{equation}{section}

\def\XXint#1#2#3{{\setbox0=\hbox{$#1{#2#3}{\int}$ }
\vcenter{\hbox{$#2#3$ }}\kern-.6\wd0}}

\date{}

\newcommand{\pat}{\partial_t}
\newcommand{\pax}{\partial_{x_2}}

\newcommand{\dxdt}{\ {\rm d}x{\rm d}t}

\newcommand{\abs}[1]{\ensuremath{\left| #1 \right|}}
\newcommand{\ep}{\varepsilon}
\newcommand{\id}{\mathbb{I}}
\newcommand{\Sym}{\mathcal{S}}

\newcommand{\R}{\mathbb{R}}
\newcommand{\tu}{\mathbb{U}}
\newcommand{\Div}{{\rm div}_x} 
\newcommand{\vr}{\varrho}
 
\newcommand{\vc}[1]{{\bf #1}}
\newcommand{\vcc}{\vc{c}}
\newcommand{\vv}{\vc{v}}
\newcommand{\vm}{\vc{m}}

\newcommand{\Grad}{\nabla_x}
\newcommand{\dx}{\,{\rm d}x}
\newcommand{\dt}{\,{\rm d}t}

\newcommand{\bfphi}{\boldsymbol{\varphi}}

\newcommand{\sil}{\rightarrow}

\newtheorem{theorem}{Theorem}[section]
\newtheorem{lemma}[theorem]{Lemma}
\newtheorem{definition}[theorem]{Definition}
\newtheorem{proposition}[theorem]{Proposition}

\title{Non--uniqueness of delta shocks and contact discontinuities in the multi--dimensional model of Chaplygin gas}
\author{Jan B\v rezina$^1$\footnote{h.brezina@gmail.com} \and  Ond\v rej Kreml$^2$\footnote{kreml@math.cas.cz} \and
V\'aclav M\'acha$^2$\footnote{macha@math.cas.cz}}

\begin{document}
\maketitle

\bigskip

\centerline{$^1$Tokyo Institute of Technology}

\centerline{2-12-1 Ookayama, Meguro-ku, Tokyo, 152-8550, Japan} 

\centerline{$^2$Institute of Mathematics of the Czech Academy of Sciences}

\centerline{\v Zitn\' a 25, CZ-115 67 Praha 1, Czech Republic}
\medskip

\bigskip

\textbf{Abstract:} We study the Riemann problem for the isentropic compressible Euler equations in two space dimensions with the pressure law describing the Chaplygin gas. It is well known that there are Riemann initial data for which the 1D Riemann problem does not have a classical $BV$ solution, instead a $\delta$-shock appears, which can be viewed as a generalized measure--valued solution with a concentration measure in the density component. We prove that in the case of two space dimensions there exists infinitely many bounded admissible weak solutions starting from the same initial data. Moreover, we show the same property also for a subset of initial data for which the classical 1D Riemann solution consists of two contact discontinuities. As a consequence of the latter result we observe that any criterion based on the principle of maximal dissipation of energy will not pick the classical 1D solution as the physical one.

\section{Introduction} 
We consider the isentropic compressible Euler system in the whole two--dimensional space, i.e. 
\begin{equation}
	\begin{split}
		\partial_t \vr + \Div (\vr\vv) &= 0, \\
		\partial_t (\vr\vv) + \Div(\vr\vv\otimes \vv)+\Grad p(\vr)  &= 0, \\
		\vr(0,\cdot) &= \vr^0, \\
		\vv(0,\cdot) &= \vv^0,
	\end{split}
	\label{eq:euler}
\end{equation}
where $\vr(t,x):[0,\infty)\times\R^2\sil\R^+$ is the unknown density and $\vv(t,x):[0,\infty)\times\R^2\sil\R^2$ the unknown velocity field. Throughout this paper we use the  notation $\vv = (v_1,v_2)$ for the components of the velocity  and $x=(x_1,x_2)$ for the space variable. We work with the pressure law describing the Chaplygin gas
\begin{equation}\label{eq:chap}
    p(\vr) = -\frac{1}{\vr}.
\end{equation}
This model has been introduced by Chaplygin \cite{chap}, see also \cite{tsien} or \cite{vK}, as an approximation for the calculation of a lifting force on a wing of a plane. Recently, the equation \eqref{eq:chap} was also considered as a suitable model for dark energy, see for example \cite{goetal}.

Note that \eqref{eq:chap} satisfies the standard condition $p' > 0$ which guarantees the hyperbolicity of the system of equations \eqref{eq:euler}. It is well known that solutions to hyperbolic systems of conservation laws may develop singularities even if the initial conditions are smooth, therefore it is reasonable to work with weaker notions of solutions. On the other hand, weak solutions may be non--unique and some admissibility conditions have to be considered. A natural way motivated by the 1D theory and the theory of scalar conservation laws is to use the entropy inequality. In the case of system \eqref{eq:euler}, the only mathematical entropy is actually the total energy $\eta = \vr\ep(\vr) + \vr \frac{\abs{\vv}^2}{2}$, where $\ep(\vr)$ is the internal energy related to the pressure $p(\vr)$ through $p(\vr) = \vr^2\ep'(\vr)$. In the case of the pressure law \eqref{eq:chap} we get
\begin{equation}\label{eq:ep}
    \ep(\vr) = \frac{1}{2\vr^2}.
\end{equation}
In this paper we call solutions {\em admissible} if they satisfy the entropy inequality
\begin{equation}\label{eq:energy}
    \pat\left(\vr\ep(\vr) + \vr \frac{\abs{\vv}^2}{2}\right) + \Div\left(\left(\vr\ep(\vr) + \vr \frac{\abs{\vv}^2}{2} + p(\vr)\right)\vv\right) \leq 0,
\end{equation}
which can be also viewed  as a form of energy balance.

We will focus on the study of the Riemann problem, i.e. the problem with the initial data
\begin{equation}\label{eq:Riem}
    (\vr^0,\vv^0)(x) = \left\{\begin{aligned}
        &(\vr_-,\vv_-) \qquad \text{ for } x_2 < 0 \\
        &(\vr_+,\vv_+) \qquad \text{ for } x_2 > 0,
    \end{aligned}\right.
\end{equation}
with $\vr_\pm, \vv_\pm$ constant, $\vr_\pm > 0$. As the initial data \eqref{eq:Riem} are one--dimensional, one can use the classical 1D theory to construct one--dimensional solutions and these solutions are further referenced as {\em classical 1D solutions}. This paper studies the question of non--uniqueness of such solutions. 

The groundbreaking theory of De Lellis and Sz\' ekelyhidi \cite{DLSz1}, \cite{DLSz2} developed for the incompressible Euler system yielded several applications in the compressible world. Already in \cite{DLSz2} the authors proved the existence of bounded initial data $(\vr^0,\vv^0)$ for which there exists infinitely many admissible weak solutions to the system \eqref{eq:euler}. Chiodaroli \cite{Ch} and Feireisl \cite{Fe} proved that such initial data can be taken with smooth density and later Chiodaroli, De Lellis and Kreml \cite{ChDLKr} proved non--uniqueness for Lipschitz initial data with the pressure law $p(\vr) = \vr^2$. Here the authors in particular used the Riemann problem as a building block in their proof.

This motivated further studies of uniqueness and non--uniqueness of bounded admissible weak solutions to the Riemann problem for \eqref{eq:euler} with $p(\vr) = \vr^\gamma$ with $\gamma \geq 1$. Uniqueness of the 1D solution in the case when the 1D solution consists only of rarefaction waves was proved by Chen and Chen \cite{Chen} (see also \cite{FeiKre}), whereas non--uniqueness was addressed in the series of papers \cite{ChiKre14}, \cite{ChiKre17}, \cite{MarKli17}, \cite{BrChKr} with the final result being non--uniqueness for all Riemann initial data allowing for 1D solution containing a shock. It is interesting to mention that the question whether a stationary solution with the initial data $\vr_-=\vr_+ = 1$, $\vv_\pm = (\pm 1,0)$ is unique or not is still open, this solution consists of a single contact discontinuity created by a discontinuity in the first component of the velocity. In the incompressible case it was shown by Sz\' ekelyhidi \cite{sz} that such solution is not unique.

We also mention the result of Klingenberg and Markfelder \cite{MaKl2}, where the authors proved that the system \eqref{eq:euler} with $p(\rho) = \rho^\gamma$, $\gamma \geq 1$ may have infinitely solutions satisfying the energy equality, i.e. \eqref{eq:energy} holds with the equality sign. This result is again based on the study of the Riemann problem. Recently, non--uniqueness of admissible weak solutions was also proved for the Riemann problem for the full Euler system (including temperature) in the case when the 1D solution contains two shocks, see \cite{AlKlKrMaMa}.

It is well known that for a certain range of Riemann initial data \eqref{eq:Riem} the Riemann problem for the Chaplygin gas does not possess a classical bounded BV solution and instead there is a solution in the form of a $\delta$-shock, which can be viewed as a generalized measure--valued solution. One of the main results of this paper shows that when the two--dimensional system of equations is considered, there exists infinitely many bounded admissible weak solutions for initial data allowing the formation of a $\delta$-shock. We also prove that in some cases when the classical 1D solution consists of two contact discontinuities, there also exists infinitely many admissible weak solutions. 

The latter result has interesting consequences with respect to the admissibility criteria based on the concept of maximal dissipation of energy. We discuss this issue in the final section.

Our main theorems follow.

\begin{theorem}\label{t:main}
Let $p(\vr) = -\vr^{-1}$. For every Riemann initial data \eqref{eq:Riem} such that the classical 1D solution consists of a $\delta$-shock, that is whenever
\begin{equation}\label{eq:deltacond}
    v_{-2} - v_{+2} \geq \frac{1}{\vr_-} + \frac{1}{\vr_+},
\end{equation}
there exists infinitely many bounded admissible weak solutions to \eqref{eq:euler}, \eqref{eq:Riem}. If, moreover, $\vr_- = \vr_+$, then there exists infinitely many bounded admissible weak solutions satisfying the energy equality.
\end{theorem}

\begin{theorem}\label{t:main2}
Let $p(\vr) = -\vr^{-1}$. Assume that the Riemann initial data \eqref{eq:Riem} satisfy
\begin{equation}\label{eq:Th2cond}
    \max\left\{\frac{1}{\vr_-}, \frac{1}{\vr_+}\right\} < v_{-2} - v_{+2} < \frac{1}{\vr_-} + \frac{1}{\vr_+}.
\end{equation}
Then there exists infinitely many bounded admissible weak solutions to \eqref{eq:euler}, \eqref{eq:Riem}. If, moreover, $\vr_- = \vr_+$, then there exists infinitely many bounded admissible weak solutions satisfying the energy equality.
\end{theorem}

The paper is structured as follows. In Section \ref{s:Riem} we recall the theory of classical 1D solutions to the Riemann problem for Chaplygin gas. In Section \ref{s:Subs} we introduce all the necessary material for the proofs of Theorems \ref{t:main} and \ref{t:main2}. In Section \ref{s:Proof1} we prove Theorem \ref{t:main} and in Section \ref{s:Proof2} we prove Theorem \ref{t:main2}. Finally, in Section \ref{s:Rem} we discuss the consequences of Theorem \ref{t:main2} with respect to two admissibility criteria based on the concept of maximal dissipation of energy.

\section{1D solutions to the Riemann problem}\label{s:Riem}

In this section we present the classical theory of one--dimensional solutions to the Riemann problem for system \eqref{eq:euler} with the Chaplygin gas pressure law \eqref{eq:chap}. For more details about the general theory of self-similar BV solutions we refer for example to \cite{Dafermos16}, the theory for the 1D equations for Chaplygin gas including $\delta$-shocks can be found for example in \cite[Section 2.1]{guo}.

To start this investigation it is reasonable to study the system in conservative variables. We search for functions $\vr(t,x_2)$, $\vm(t,x_2) = (m_1,m_2)(t,x_2) = (\vr v_1,\vr v_2)(t,x_2)$ that depend in the space variable only on $x_2$ and solve the two--dimensional isentropic Euler system \eqref{eq:euler} formulated as
\begin{equation}\label{eq:1DRiem}
\begin{split}
    \pat \vr + \pax m_2 &= 0, \\
    \pat m_1 + \pax \frac{m_1m_2}{\vr} &= 0, \\
    \pat m_2 + \pax \left(\frac{m_2^2}{\vr} - \frac{1}{\vr}\right) &= 0
\end{split}    
\end{equation}
with the initial data
\begin{equation}\label{eq:1DRiemIC}
(\vr^0(x),\vm^0(x)) = \left\{
    \begin{aligned}
        &(\vr_-,\vm_-) = (\vr_-,\vr_-\vv_-) \quad \text{ for } x_2 < 0 \\
        &(\vr_+,\vm_+) = (\vr_+,\vr_+\vv_+) \quad \text{ for } x_2 > 0.
    \end{aligned}\right.
\end{equation}

The state vector is thus defined as $U = (\vr,\vm)^T$ and the flux is $F(U) = (m_2,\frac{m_1m_2}{\vr},\frac{m_2^2}{\vr} - \frac{1}{\vr} )^T$. The eigenvalues of the Jacobian matrix of $F(U)$ are 
\begin{equation}\label{eq:eigenval}
    \lambda_1 = \frac{m_2}{\vr} - \frac{1}{\vr}, \qquad \lambda_2 = \frac{m_2}{\vr}, \qquad \lambda_3 = \frac{m_2}{\vr} + \frac{1}{\vr}
\end{equation}
and the right eigenvectors are
\begin{equation}\label{eq:eigenvec}
    R_1 = \left(
    \begin{array}{c}
                1 \\ 
        \frac{m_1}{\vr} \\
        \frac{m_2}{\vr} - \frac{1}{\vr}
    \end{array}\right), \qquad    R_2 = \left(
    \begin{array}{c}
        0 \\
        1 \\
        0
    \end{array}\right), \qquad
    R_3 = \left(
    \begin{array}{c}
        1 \\
        \frac{m_1}{\vr} \\
        \frac{m_2}{\vr} + \frac{1}{\vr}
    \end{array}\right).
\end{equation}

Since we observe that $\nabla\lambda_i\cdot R_i = 0$ for all $i = 1,2,3$, we conclude that all three characteristic families of the system \eqref{eq:1DRiem} are linearly degenerate and therefore all elementary waves are contact discontinuities. Moreover, the variable $m_1$ appears only in the equation \eqref{eq:1DRiem}$_2$ and thus the system can be decoupled. In particular, if we consider bounded solutions, we first solve the system \eqref{eq:1DRiem}$_1$, \eqref{eq:1DRiem}$_3$ and to its solution we add a contact discontinuity with the speed $\lambda_2 = \frac{m_2}{\vr} = v_2$, where the velocity component $v_1$ jumps from $v_{-1}$ to $v_{+1}$.

As we already observed, bounded solutions of \eqref{eq:1DRiem}$_1$, \eqref{eq:1DRiem}$_3$ are contact discontinuities with the speeds $\lambda_1$ or $\lambda_3$, satisfying the Rankine-Hugoniot conditions 
\begin{equation}\label{eq:RHgeneral}
    \begin{split}
        \sigma[\vr] &= [\vr v_2], \\
        \sigma[\vr v_2] &= [\vr v_2^2 - \frac{1}{\vr}],
    \end{split}
\end{equation}
where $[q] = q_R - q_L$. Here, the index $R$ denotes the constant state on the right and $L$ denotes the constant state on the left, and $\sigma$ is the speed of the contact discontinuity. We easily obtain for the $1$-characteristic family that
\begin{equation}
\sigma = \lambda_1 = v_{2R} - \frac{1}{\vr_R} = v_{2L} - \frac{1}{\vr_L}
\end{equation}
and for the $3$-characteristic family that
\begin{equation}
\sigma = \lambda_3 = v_{2R} + \frac{1}{\vr_R} = v_{2L} + \frac{1}{\vr_L}.
\end{equation}

With this information we can draw the characteristic curves starting from a point $(\vr_-,v_{-2})$ in the state space as
\begin{align}
    v_2 - \frac{1}{\vr} &= v_{-2} - \frac{1}{\vr_-} \label{eq:curve1},    \\
v_2 + \frac{1}{\vr} &= v_{-2} + \frac{1}{\vr_-}   \label{eq:curve2} 
\end{align}
 and conclude the following:
\begin{itemize}
    \item If $(\vr_+,v_{+2}) = (\vr_-,v_{-2})$, then  this constant state is a solution.
    \item If $(\vr_+,v_{+2})$ lies on one of the curves \eqref{eq:curve1}, \eqref{eq:curve2}, then the solution consists of one contact discontinuity.
    \item If $(\vr_+,v_{+2})$ does not lie on neither of the curves \eqref{eq:curve1}, \eqref{eq:curve2} and it holds
    \begin{equation}\label{eq:2CDcond}
        v_{-2} - v_{+2} < \frac{1}{\vr_-} + \frac{1}{\vr_+},
    \end{equation}
    then the solution consists of two contact discontinuities with the intermediate constant state $(\vr_m,v_{m2})$ given by
    \begin{equation}\label{eq:middlestate}
        \frac{1}{\vr_m} = \frac 12 \left(v_{+2}-v_{-2}\right) + \frac 12\left(\frac{1}{\vr_+} + \frac{1}{\vr_-}\right), \qquad v_{m2} =  \frac 12 \left(v_{+2}+v_{-2}\right) + \frac 12\left(\frac{1}{\vr_+} - \frac{1}{\vr_-}\right).
    \end{equation}
    \item In the case 
    \begin{equation}\label{eq:deltacond00}
        v_{-2} - v_{+2} \geq \frac{1}{\vr_-} + \frac{1}{\vr_+}
    \end{equation}
    there does not exist a self-similar BV solution to the Riemann problem \eqref{eq:1DRiem}$_1$, \eqref{eq:1DRiem}$_3$ with the initial data \eqref{eq:1DRiemIC}.
\end{itemize}

We will now discuss the last case in detail. We introduce solutions with Dirac delta measures supported at a jump, i.e. we search for a solution in the form
\begin{equation}
    (\vr,v_1,v_2)(t,x_2) = \left\{\begin{aligned}
        &(\vr_-,v_{-1},v_{-2}) \qquad &\text{ for } x_2 < \sigma t \\
        &(\omega(t)\delta_{x_2-\sigma t},\xi,\sigma) \qquad &\text{ for } x_2 = \sigma t \\
        &(\vr_+,v_{+1},v_{+2}) \qquad &\text{ for } x_2 > \sigma t,
    \end{aligned}\right.
\end{equation}
where the $\delta$-shock speed $\sigma$ is also the value of the second component of the velocity on the $\delta$-shock, $\xi$ is the value of the first component of the velocity on the $\delta$-shock, and together with the density weight $\omega(t)$ they satisfy the generalized Rankine-Hugoniot conditions
\begin{align}
    \frac{\mathrm{d} \omega(t)}{\dt} &= \sigma(\vr_+-\vr_-) - (\vr_+v_{+2} - \vr_-v_{-2}), \label{eq:deltaeq1}\\
    \frac{\mathrm{d} \omega(t)\xi}{\dt} &= \sigma(\vr_+v_{+1}-\vr_-v_{-1}) - (\vr_+v_{+1}v_{+2} - \vr_-v_{-1}v_{-2}), \label{eq:deltaeq15} \\
    \frac{\mathrm{d} \omega(t)\sigma}{\dt} &= \sigma(\vr_+v_{+2}-\vr_-v_{-2}) - \left(\vr_+v_{+2}^2 - \vr_-v_{-2}^2 - \frac{1}{\vr_+} + \frac{1}{\vr_-}\right). \label{eq:deltaeq2}
\end{align}
Here we set the initial data $\omega(0) = 0$ and we define $\frac{1}{\vr} = 0$ for $x_2 = \sigma t$. This yields the following solution to \eqref{eq:deltaeq1}-\eqref{eq:deltaeq2}
\begin{align} \label{eq:omega1}
    \omega(t) &= \sqrt{\vr_+\vr_-\left((v_{+2} - v_{-2})^2 - \left(\frac{\vr_- - \vr_+}{\vr_-\vr_+} \right)^2\right)}t, \\
    \sigma &= \frac{\vr_+v_{+2} - \vr_-v_{-2} + \sqrt{\vr_+\vr_-\left((v_{+2} - v_{-2})^2 - \left(\frac{\vr_- - \vr_+}{\vr_-\vr_+}\right)^2\right)}}{\vr_+-\vr_-} \label{eq:sigma1}, \\
    \xi &= \frac{(\vr_+v_{+1}-\vr_-v_{-1})\sqrt{\vr_+\vr_-\left((v_{+2} - v_{-2})^2 - \left(\frac{\vr_- - \vr_+}{\vr_-\vr_+}\right)^2\right)} + \vr_-\vr_+(v_{+2}-v_{-2})(v_{+1}-v_{-1})}{(\vr_+-\vr_-)\sqrt{\vr_+\vr_-\left((v_{+2} - v_{-2})^2 - \left(\frac{\vr_- - \vr_+}{\vr_-\vr_+}\right)^2\right)}} \label{eq:xi1}
\end{align}
in the case $\vr_+ \neq \vr_-$ and
\begin{align}
\omega(t) &= (\vr_-v_{-2} - \vr_+v_{+2})t, \label{eq:omega2}\\
\sigma &= \frac 12 (v_{+2} + v_{-2}), \label{eq:sigma2} \\
\xi &= \frac 12 (v_{+1} + v_{-1}) \label{eq:xi2}
\end{align}
in the case $\vr_+ = \vr_-$. Note that an easy calculation shows $v_{+2} < \sigma < v_{-2}$ and similarly one can show that $\min\{v_{-1},v_{+1}\} \leq \xi \leq \max\{v_{-1},v_{+1}\}$ with equalities if and only if $v_{-1} = v_{+1}$.

Next we specify the sense in which the Euler system is satisfied in the case of a $\delta$-shock. 
Such solution can be viewed as a generalized measure--valued solution $(\nu_{t,x}, m_\vr)$, where $\nu_{t,x} \in L^\infty_{w}((0,\infty)\times \R^2; \mathcal{P}(\R^+ \times \R^2))$ is a standard atomic probability measure defined simply as 
\begin{equation}\label{eq:nutx}
    \nu_{t,x} = \left\{
    \begin{aligned}
        \delta_{(\vr_-,\vv_{-})} \qquad \text{ for } x_2 < \sigma t \\
        \delta_{(\vr_+,\vv_{+})} \qquad \text{ for } x_2 > \sigma t
    \end{aligned}
        \right.
\end{equation}
and $m_\vr \in \mathcal{M}([0,\infty)\times \R^2)$ is a concentration measure defined by 
\begin{equation}\label{eq:murho}
   m_\vr = \omega(t)\delta_{x_2-\sigma t}\,\mathrm{d} x_1 \dt.
\end{equation}
In particular, the duality between this measure and a smooth function $\varphi$ is defined as
\begin{equation}\label{eq:duality}
    \langle m_\vr,\varphi\rangle = \int_0^\infty \int_\R \omega(t)\varphi(t,x_1,\sigma t)\, \mathrm{d} x_1 \dt.
\end{equation}

If we denote
\begin{equation}\label{eq:nutx2}
    (\vr,\vv) = \left\{
    \begin{aligned}
        &(\vr_-,\vv_{-}) \qquad \text{ for } x_2 < \sigma t \\
        &(\vr_+,\vv_{+}) \qquad \text{ for } x_2 > \sigma t,
    \end{aligned}
        \right.
\end{equation}
then the equations \eqref{eq:euler} are satisfied in the sense
\begin{align} \nonumber
    &\int_0^\infty \int_{\R^2} \vr\pat\varphi \dxdt + \langle m_\vr,\pat\varphi \rangle +  \int_0^\infty \int_{\R^2} \vr\vv\cdot\nabla_x\varphi \dxdt + \langle m_\vr,(\xi,\sigma)\cdot\nabla_x \varphi\rangle \\
    & = -\int_{\R^2} \vr^0(x)\varphi(0,x) \dx
\end{align}
for all $\varphi \in C^\infty_c([0,\infty)\times \R^2)$ and
\begin{align} \nonumber
    &\int_0^\infty \int_{\R^2} \vr\vv \cdot\pat\bfphi \dxdt + \langle m_\vr,(\xi,\sigma)\cdot\pat\bfphi \rangle \\
    &+ \int_0^\infty \int_{\R^2} \left(\vr\vv\otimes \vv - \frac{1}{\vr} \right):\nabla_x\bfphi \dxdt + \langle m_\vr,(\xi,\sigma)\otimes(\xi,\sigma):\nabla_x \bfphi\rangle \nonumber \\
    & = -\int_{\R^2} \vr^0(x)\vv^0(x)\cdot\bfphi(0,x) \dx
\end{align}
for all $\bfphi \in C^\infty_c([0,\infty)\times \R^2)$.

We close this section with examining the validity of the energy inequality \eqref{eq:energy} for the 1D solutions mentioned above. It is not difficult to observe that in the case of a solution consisting only of contact discontinuities, the energy inequality \eqref{eq:energy} (more precisely its weak formulation) holds as an equality. In the case of $\delta$-shock we have the following Lemma.
\begin{lemma}\label{l:deltaenergy}
Let $(\vr_\pm,\vv_\pm)$ are Riemann initial data such that \eqref{eq:deltacond00} holds. Assume that $\omega(t)$, $\sigma$ and $\xi$ are given by \eqref{eq:omega1}-\eqref{eq:xi1} in the case $\vr_- \neq \vr_+$ and by \eqref{eq:omega2}-\eqref{eq:xi2} in the case $\vr_- = \vr_+$. Then the couple $(\nu_{t,x},m_\vr)$ defining the $\delta$-shock solution given as in \eqref{eq:nutx}, \eqref{eq:murho} satisfies the following version of the energy inequality
\begin{align}
    \nonumber
    &\int_0^\infty \int_{\R^2} \left(\frac{1}{2\vr} + \vr\frac{|\vv|^2}{2}\right) \pat\varphi \dxdt + \langle m_\vr,\frac{\xi^2 + \sigma^2}{2}\pat\varphi \rangle \\
    &+ \int_0^\infty \int_{\R^2} \left(\vr\frac{|\vv|^2}{2} - \frac{1}{2\vr} \right)\vv\cdot\nabla_x\varphi \dxdt + \langle m_\vr,\frac{\xi^2+\sigma^2}{2}(\xi,\sigma) \cdot\nabla_x \varphi\rangle \nonumber \\
    & \geq -\int_{\R^2} \left(\frac{1}{2\vr^0(x)} + \vr^0(x)\frac{|\vv^0(x)|^2}{2}\right)\varphi(0,x) \dx \label{eq:enin}
\end{align}
for all $\varphi \in C^\infty_c([0,\infty)\times\R^2)$, $\varphi \geq 0$.
\end{lemma}
In order to prove Lemma \ref{l:deltaenergy} it is useful to introduce the Galilean transformation for $\delta$-shock solutions.

\begin{lemma}\label{lem:galileo}
Let $(\vr, \vv)$ and $m_\vr$ be a $\delta$-shock solution to \eqref{eq:euler}, \eqref{eq:Riem} as defined in \eqref{eq:nutx}, \eqref{eq:murho} and \eqref{eq:nutx2}. Then for any $\vcc = (c_1,c_2) \in \R^2$
\begin{equation}\label{eq:tildaveci}
\begin{split}
    \tilde{\vr}(t,x) &:= \vr(t,x-\vcc t), \\ 
    \tilde{\vv}(t,x) &:= \vv(t,x-\vcc t) + \vcc, \\
    \tilde{\omega}(t) &:= \omega(t), \\
    \tilde{\sigma} &:= \sigma + c_2, \\
    \tilde{\xi} &:= \xi + c_1, \\
    \widetilde{m_\vr} &:= \tilde{\omega}(t)\delta_{x_2-\tilde{\sigma} t}\,\mathrm{d} x_1 \dt
\end{split}
\end{equation}
is a $\delta$-shock solution to \eqref{eq:euler} with the initial data $(\tilde{\vr}^0(x),\tilde{\vv}^0(x)) = (\vr^0(x),\vv^0(x) + \vcc)$.
\end{lemma}
\begin{proof}
The Galilean transformation is a standard tool in the framework of weak solutions for Euler equations. Here we simply need to verify the Rankine-Hugoniot conditions \eqref{eq:deltaeq1}--\eqref{eq:deltaeq2} with $\tilde{\omega}$ and $\tilde{\sigma}$ defined in \eqref{eq:tildaveci}. We recall that for two arbitrary functions $f$ and $g$ it holds $[f(g+c)] = [fg] + c[f]$ for any constant $c\in \mathbb R$.
The first condition \eqref{eq:deltaeq1} rewrites as
\begin{equation*}
\frac{{\rm d} \tilde \omega(t)}{{\rm d}t} = \frac{{\rm d} \omega(t)}{{\rm d}t} = \sigma [\vr] - [\vr v_2] = \tilde \sigma[\tilde\vr] - c_2[\tilde\vr] - [\tilde\vr \tilde v_2] + c_2[\tilde \vr] = \tilde \sigma[\tilde \vr] - [\tilde \vr\tilde v_2].
\end{equation*}
This relation is then used in the derivation of the second condition \eqref{eq:deltaeq15} since
\begin{align*}
&\frac{{\rm d}\tilde \omega(t) \tilde \xi}{{\rm d}t} = \frac{{\rm d} \omega (t)  \xi}{{\rm d}t} + c_1 \frac{{\rm d}\tilde \omega(t)}{{\rm d}t} = \sigma[\vr v_1] - [\vr v_1 v_2] + c_1\tilde \sigma[\tilde \vr] - c_1[\tilde \vr\tilde v_2] \\ 
& \quad = \tilde \sigma[\tilde \vr \tilde v_1] - c_1\tilde \sigma[\tilde \vr] -c_2[\tilde\vr\tilde v_1] + c_1 c_2[\tilde \vr] - [\tilde \vr \tilde v_1 \tilde v_2] + c_1 [\tilde \vr \tilde v_2] + c_2 [\tilde \vr \tilde v_1] - c_1 c_2[\tilde \vr] + c_1\tilde \sigma[\tilde \vr] - c_1[\tilde \vr\tilde v_2] \\
&\quad = \tilde \sigma [\tilde \vr \tilde v_1] - [\tilde \vr \tilde v_1 \tilde v_2]
\end{align*}
and the third condition \eqref{eq:deltaeq2} since
\begin{align*}
&\frac{{\rm d}\tilde \omega(t) \tilde \sigma}{{\rm d}t} = \frac{{\rm d} \omega (t)  \sigma}{{\rm d}t} + c_2 \frac{{\rm d}\tilde \omega(t)}{{\rm d}t} = \sigma[\vr v_2] - [\vr v_2^2] + [\vr^{-1}] + c_2\tilde \sigma[\tilde \vr] - c_2[\tilde \vr\tilde v_2] \\ 
& \quad = \tilde \sigma[\tilde \vr \tilde v_2] - c_2\tilde \sigma[\tilde \vr] -c_2[\tilde\vr\tilde v_2] + c_2^2[\tilde \vr] - [\tilde \vr \tilde v_2^2] + 2c_2 [\tilde \vr \tilde v_2] - c_2^2[\tilde \vr] + [\vr^{-1}] + c_2\tilde \sigma[\tilde \vr] - c_2[\tilde \vr\tilde v_2] \\
&\quad = \tilde \sigma [\tilde \vr \tilde v_2] - [\tilde \vr \tilde v_2^2 - \tilde \vr^{-1}].
\end{align*}
\end{proof}

Now we can prove Lemma \ref{l:deltaenergy}.

\begin{proof}
Thanks to  Lemma \ref{lem:galileo} we may assume that $\sigma = 0$ and $\xi = 0$. This together with \eqref{eq:deltaeq15}-\eqref{eq:deltaeq2} yield the relations
\begin{align}
\vr_+v_{+1}v_{+2} - \vr_-v_{-1}v_{-2} &= 0, \label{eq:eninprv}\\ 
\vr_+v_{+2}^2 - \vr_-v_{-2}^2 - \frac{1}{\vr_+} + \frac{1}{\vr_-} &= 0.\label{eq:enin1}
\end{align}

Further, since $\sigma = 0$, we can split the first integral on the left of \eqref{eq:enin} and calculate
\begin{align}
&\int_0^\infty \int_{\R^2} \left(\frac{1}{2\vr} + \vr\frac{|\vv|^2}{2}\right) \pat\varphi \dxdt \nonumber \\ 
&\quad = \int_0^\infty \int_{x_2 < 0}  \left(\frac{1}{2\vr} + \vr\frac{|\vv|^2}{2}\right) \pat\varphi \dxdt + \int_0^\infty \int_{x_2 > 0} \left(\frac{1}{2\vr} + \vr\frac{|\vv|^2}{2}\right) \pat\varphi \dxdt \nonumber \\
&\quad = - \int_{x_2 < 0}  \left(\frac{1}{2\vr^0} + \vr^0\frac{|\vv^0|^2}{2}\right)\varphi(0,x) \dx -  \int_{x_2 > 0} \left(\frac{1}{2\vr^0} + \vr^0\frac{|\vv^0|^2}{2}\right)\varphi(0,x) \dx \nonumber \\
&\quad = - \int_{\R^2}  \left(\frac{1}{2\vr^0} + \vr^0\frac{|\vv^0|^2}{2}\right) \dx.
\end{align}
Moreover, since $\xi = \sigma = 0$, the duality pairings with the measure $m_\vr$ in \eqref{eq:enin} are trivially zero. So to prove  \eqref{eq:enin} it is enough to show that 
\begin{equation}\label{eq:enin00}
    \int_0^\infty \int_{\R^2} \left(\vr\frac{|\vv|^2}{2} - \frac{1}{2\vr} \right)\vv\cdot\nabla_x\varphi \dxdt = \int_0^\infty \int_{\R^2} \left(\vr\frac{v_1^2 + v_2^2}{2} - \frac{1}{2\vr} \right)(v_1\,\partial_{x_1}\varphi + v_2\,\pax\varphi) \dxdt \geq 0.
\end{equation}
To do so, we split once again the integral over $\R^2$ into the negative and the positive $x_2$-half-planes and observe that 
\begin{equation}
\int_0^\infty \int_{\pm x_2 > 0} \int_\R  \left(\vr\frac{v_1^2 + v_2^2}{2} - \frac{1}{2\vr} \right) v_1 \partial_{x_1} \varphi\, \mathrm{d}x_1\, \dx_2\,\dt = 0  . 
\end{equation}
Therefore, to get \eqref{eq:enin00} we must show that
\begin{equation}
\vr_+v_{+2}^3 - \vr_-v_{-2}^3 - \frac{v_{+2}}{\vr_+} + \frac{v_{-2}}{\vr_-} + \vr_+v_{+1}^2v_{+2} - \vr_- v_{-1}^2 v_{-2}\leq 0.\label{eq:enin21}
\end{equation}
First, we observe that $\vr_+v_{+1}^2v_{+2} - \vr_- v_{-1}^2 v_{-2}\leq 0$. Indeed, using \eqref{eq:eninprv} we get that the expression on the left-hand side is equal to $(v_{+1} - v_{-1})\vr_+v_{+1}v_{+2}$ and we already know that $\vr_+ > 0$, the terms $v_{+1}$ and $v_{+1}-v_{-1}$ have the same sign and $v_{+2} < \sigma = 0$.

Second, we prove the inequality
\begin{equation}
\vr_+v_{+2}^3 - \vr_-v_{-2}^3 - \frac{v_{+2}}{\vr_+} + \frac{v_{-2}}{\vr_-} \leq 0\label{eq:enin2}
\end{equation}
by working backwards. 
Since  $2\eqref{eq:enin2} - v_{+2}\eqref{eq:enin1} - v_{-2}\eqref{eq:enin1}$ yields
\begin{equation*}
\vr_+v_{+2}^3 - \vr_-v_{-2}^3 - \frac{v_{+2}}{\vr_+} + \frac{v_{-2}}{\vr_-} + \vr_- v_{-2}^2v_{+2} - \vr_+ v_{+2}^2 v_{-2} + \frac{v_{-2}}{\vr_+} - \frac{v_{+2}}{\vr_-} \leq 0,
\end{equation*}
we get that
\begin{equation}\label{eq:enin3}
(v_{-2} - v_{+2}) \left(\frac 1{\vr_{+}} - \vr_{+2} v_{+2}^2\right) + (v_{+2}-v_{-2})\left(\vr_-v_{-2}^2 - \frac 1{\vr_-}\right)\leq 0
\end{equation}
is equivalent to \eqref{eq:enin2}. Furthermore, as $v_{-2}- v_{+2} > 0$ we deduce that \eqref{eq:enin3} is equivalent to
\begin{equation}\label{eq:enin4}
\vr_+ v_{+2}^2 + \vr_- v_{-2}^2 - \frac 1{\vr_-} - \frac1{\vr_+} \geq 0.
\end{equation}
Next, summing $\eqref{eq:enin4} + \eqref{eq:enin1}$,  subtracting $\eqref{eq:enin4} - \eqref{eq:enin1}$, respectively, we obtain that
$$
v_{+2}^2 \geq \frac1{\vr_+^2}, \qquad \mbox{ resp. } v_{-2}^2 \geq \frac 1{\vr_-^2}
$$
has to be satisfied. Since we know that $v_{+2} < \sigma = 0 < v_{-2}$ the above relations are equivalent to
\begin{equation}
    v_{+2} \leq - \frac{1}{\vr_+}, \qquad \mbox{ resp. } v_{-2} \geq \frac 1{\vr_-}. \label{eq:finaltouch}
\end{equation}
Finally, we use \eqref{eq:enin1} one more time in the form
\begin{equation}
    \vr_+\left(-v_{+2}+\frac{1}{\vr_+}\right)\left(-v_{+2}-\frac{1}{\vr_+}\right) = \vr_-\left(v_{-2}+\frac{1}{\vr_-}\right)\left(v_{-2}-\frac{1}{\vr_-}\right).
\end{equation}
The first two terms in the products on both sides are clearly positive and hence $-v_{+2}-\frac{1}{\vr_+}$ and $v_{-2}-\frac{1}{\vr_-}$  have the same sign or are both equal to $0$. Since the sum of these terms is non--negative, see \eqref{eq:deltacond00}, both terms are non--negative and \eqref{eq:finaltouch} holds true.
This concludes the proof.

\end{proof}

\section{Subsolutions}\label{s:Subs}

The proofs of Theorems \ref{t:main} and \ref{t:main2} are based on the notion of admissible fan subsolutions introduced originally in \cite{ChDLKr} and on the construction of infinitely many solutions related to a single subsolution in the spirit of \cite{DLSz1}, \cite{DLSz2}. Therefore we first introduce the necessary definitions.

\begin{definition}[Fan partition]\label{d:fan}
A {\em fan partition} of $\R^2\times (0, \infty)$ consists of open sets $P_-, P_1, P_2, P_+$
of the following form 
\begin{align}
 P_- &= \{(t,x): t>0 \quad \mbox{and} \quad x_2 < \nu_- t\}\\
 P_1 &= \{(t,x): t>0 \quad \mbox{and} \quad \nu_- t < x_2 < \nu_0 t\}\\
 P_2 &= \{(t,x): t>0 \quad \mbox{and} \quad \nu_0 t < x_2 < \nu_+ t\}\\
 P_+ &= \{(t,x): t>0 \quad \mbox{and} \quad x_2 > \nu_+ t\},
\end{align}
where $\nu_- < \nu_0 < \nu_+$ is an arbitrary trio of real numbers.
\end{definition}

We denote by $\Sym_0^{2\times2}$ the set of all symmetric $2\times2$ matrices with zero trace, by $\id$ the $2\times2$ identity matrix and by $\bm{1}_P$ the indicator function of a set $P$.

\begin{definition}[Fan subsolution] \label{d:subs}
A {\em fan subsolution} to the compressible isentropic Euler system \eqref{eq:euler} with the
initial data \eqref{eq:Riem} is a triple 
$(\overline{\vr}, \overline{\vv}, \overline{\tu}): \R^2\times 
(0,\infty) \rightarrow (\R^+, \R^2, \Sym_0^{2\times2})$ of piece--wise constant functions satisfying
the following requirements.
\begin{itemize}
\item[(i)] There is a fan partition $P_-, P_1, P_2, P_+$ of $\R^2\times (0, \infty)$ such that
\[
(\overline{\vr}, \overline{\vv}, \overline{\tu})=  
(\vr_-, \vv_-, \tu_-) \bm{1}_{P_-}
+ (\vr_1, \vv_1, \tu_1) \bm{1}_{P_1}
+ (\vr_2, \vv_2, \tu_2) \bm{1}_{P_2}
+ (\vr_+, \vv_+, \tu_+) \bm{1}_{P_+},
\]
where $\vr_i, \vv_i, \tu_i$ ($i = 1,2$) are constants with $\vr_i >0$ and $\tu_\pm =
\vv_\pm\otimes \vv_\pm - \textstyle{\frac{1}{2}} |\vv_\pm|^2 \id$;
\item[(ii)] There exists a positive constants $C_i$ ($i = 1,2$) such that
\begin{equation} \label{eq:subsolution 2}
\vv_i\otimes \vv_i - \tu_i < \frac{C_i}{2} \id\, ;
\end{equation}
\item[(iii)] The triple $(\overline{\vr}, \overline{\vv}, \overline{\tu})$ solves the following system in the
sense of distributions:
\begin{align}
&\partial_t \overline{\vr} + {\rm div}_x (\overline{\vr} \, \overline{\vv}) \;=\; 0,\label{eq:continuity}\\
&\partial_t (\overline{\vr} \, \overline{\vv})+{\rm div}_x \left(\overline{\vr} \, \overline{\tu} 
\right) + \nabla_x \left( p(\overline{\vr})+\frac{1}{2}\left( \overline{\vr} |\overline{\vv}|^2 \bm{1}_{P_+\cup P_-} + \sum_{i=1}^2 C_i \vr_i
\bm{1}_{P_i}\right)\right)= 0.\label{eq:momentum}
\end{align}
\end{itemize}
\end{definition}

\begin{definition}[Admissible fan subsolution]\label{d:admiss}
 A fan subsolution $(\overline{\vr}, \overline{\vv}, \overline{\tu})$ is said to be {\em admissible}
if it satisfies the following inequality in the sense of distributions
\begin{align} 
&\pat \left(\overline{\vr} \varepsilon(\overline{\vr})\right)+{\rm div}_x
\left[\left(\overline{\vr}\varepsilon(\overline{\vr})+p(\overline{\vr})\right) \overline{\vv}\right]
 + \pat \left( \overline{\vr}\, \frac{|\overline{\vv}|^2}{2}\, \bm{1}_{P_+\cup P_-} \right)
+ {\rm div}_x \left(\overline{\vr}\, \frac{|\overline{\vv}|^2}{2}\, \overline{\vv}\, \bm{1}_{P_+\cup P_-}\right)\nonumber\\
&\qquad\qquad+ \left[\pat\left(\sum_{i=1}^2\vr_i \, \frac{C_i}{2} \, \bm{1}_{P_i}\right) 
+ {\rm div}_x\left(\sum_{i=1}^2\vr_i \, \overline{\vv} \, \frac{C_i}{2}  \, \bm{1}_{P_i}\right)\right]
\;\leq\; 0\, .\label{eq:admissible subsolution}
\end{align}
\end{definition}

A key ingredient in the proofs of main theorems of this paper is the following proposition proved in \cite{ChDLKr}.

\begin{proposition}\label{p:subs}
Let $p(\vr)$ be any $C^1$ function and $(\vr_\pm, \vv_\pm)$ be such that there exists at least one
admissible fan subsolution $(\overline{\vr}, \overline{\vv}, \overline{\tu})$ of \eqref{eq:euler}
with initial data \eqref{eq:Riem}. Then there are infinitely 
many bounded admissible solutions $(\vr, \vv)$ to \eqref{eq:euler}, \eqref{eq:Riem} such that 
$\vr=\overline{\vr}$ and $\abs{\vv}^2\bm{1}_{P_i} = C_i$, $i=1,2$.
\end{proposition}

The infinitely many bounded admissible weak solutions $(\vr,\vv)$ are constructed by adding, in a non--unique way, solutions to  the linearized pressureless incompressible Euler equations supported in sets $P_1$ and $P_2$ to the single subsolution. The procedure is described in the following lemma, which is a key building block of the proof of Proposition \ref{p:subs}, cf. \cite[Lemma 3.7]{ChDLKr}.

\begin{lemma}\label{l:ci}
Let $(\tilde{\vv}, \tilde{\tu})\in \R^2\times \Sym_0^{2\times 2}$ and $C_0>0$ be such that $\tilde{\vv}\otimes \tilde{\vv}
- \tilde{\tu} < \frac{C_0}{2} \id$. For any open set $\Omega\subset \R^2\times \R$ there are infinitely many maps
$(\underline{\vv}, \underline{\tu}) \in L^\infty (\R^2\times \R ; \R^2\times \Sym_0^{2\times 2})$ with the following property
\begin{itemize}
\item[(i)] $\underline{\vv}$ and $\underline{\tu}$ vanish identically outside $\Omega$;
\item[(ii)] ${\rm div}_x \underline{\vv} = 0$ and $\partial_t \underline{\vv} + {\rm div}_x \underline{\tu} = 0$;
\item[(iii)] $ (\tilde{\vv} + \underline{\vv})\otimes (\tilde{\vv} + \underline{\vv}) - (\tilde{\tu} + \underline{\tu}) = \frac{C_0}{2} \id$
a.e. on $\Omega$.
\end{itemize}
\end{lemma}

It is easy to see that the application of Lemma \ref{l:ci} with $\Omega = P_i$, $(\tilde{\vv}, \tilde{\tu}) = (\vv_i,\tu_i)$ and $C_0 = C_i$ yields the proof of Proposition \ref{p:subs}. One only needs to check that each couple $(\overline{\vr}, \overline{\vv} + \sum_{i=1}^2\underline{\vv}_i)$ is an admissible weak solution to \eqref{eq:euler} with the initial data \eqref{eq:Riem}. More details of the proof are available in \cite[Section 3.3]{ChDLKr}. 

\section{Proof of Theorem \ref{t:main}}\label{s:Proof1}

Using Proposition \ref{p:subs} we know that in order to prove Theorem \ref{t:main} it is enough to find a single admissible fan subsolution. Therefore let us now fix such initial data $\vr_-$, $\vr_+$, $\vv_-$ and $\vv_+$  that they allow for a 1D solution in the form of a $\delta$-shock, namely they satisfy
    \begin{equation}\label{eq:deltacond2}
        v_{-2} - v_{+2} \geq \frac{1}{\vr_-} + \frac{1}{\vr_+}.
    \end{equation}

In accordance with Definition \ref{d:subs} we have to find the interface speeds $\nu_- < \nu_0 < \nu_+$, the constant middle states $(\vr_i, \vv_i,\tu_i)$ and positive constants $C_i$, $i = 1,2$ in order to obtain an admissible fan subsolution $(\overline{\vr},\overline{\vv}, \overline{\tu})$. We denote $\vv_{i} = (\alpha_i,\beta_i)$ and
\begin{equation}\label{eq:tu1}
    \tu_i =\left( \begin{array}{cc}
    \gamma_i & \delta_i \\
    \delta_i & -\gamma_i\\
    \end{array} \right).
\end{equation}
Then the equations \eqref{eq:continuity} and \eqref{eq:momentum} translate to the following set of Rankine-Hugoniot conditions on the left interface:
\begin{align}
&\nu_- (\vr_- - \vr_1) \, =\,  \vr_- v_{-2} -\vr_1  \beta_1, \label{eq:cont_left}  \\
&\nu_- (\vr_- v_{-1}- \vr_1 \alpha_1) \, = \, \vr_- v_{-1} v_{-2}- \vr_1 \delta_1,  \label{eq:mom_1_left}\\
&\nu_- (\vr_- v_{-2}- \vr_1 \beta_1) \, = \,  
\vr_- v_{-2}^2 + \vr_1 \gamma_1 +p (\vr_-)-p (\vr_1) - \vr_1 \frac{C_1}{2}\, ;\label{eq:mom_2_left}
\end{align}
on the middle interface:
\begin{align}
&\nu_0 (\vr_1 - \vr_2) \, =\,  \vr_1  \beta_1 - \vr_2  \beta_2, \label{eq:cont_mid}  \\
&\nu_0 (\vr_1 \alpha_1 - \vr_2 \alpha_2) \, = \, \vr_1 \delta_1 - \vr_2 \delta_2,  \label{eq:mom_1_mid}\\
&\nu_0 (\vr_1 \beta_1 - \vr_2 \beta_2) \, = \,  
 \vr_1 \frac{C_1}{2} - \vr_1 \gamma_1 +p (\vr_1)- \vr_2 \frac{C_2}{2} + \vr_2 \gamma_2 -p (\vr_2) \, ;\label{eq:mom_2_mid}
\end{align}
and on the right interface:
\begin{align}
&\nu_+ (\vr_2-\vr_+ ) \, =\,  \vr_2  \beta_2 - \vr_+ v_{+2}, \label{eq:cont_right}\\
&\nu_+ (\vr_2 \alpha_2- \vr_+ v_{+1}) \, = \, \vr_2 \delta_2 - \vr_+ v_{+1} v_{+2}, \label{eq:mom_1_right}\\
&\nu_+ (\vr_1 \beta_1- \vr_+ v_{+2}) \, = \, - \vr_2 \gamma_2 - \vr_+ v_{+2}^2 +p (\vr_2) -p (\vr_+) 
+ \vr_2 \frac{C_2}{2}\, .\label{eq:mom_2_right}
\end{align} 
The  subsolution condition \eqref{eq:subsolution 2} can be rewritten as
\begin{align}
 &\alpha_i^2 +\beta_i^2 < C_i, \label{eq:sub_trace}\\
& \left( \frac{C_i}{2} -{\alpha_i}^2 +\gamma_i \right) \left( \frac{C_i}{2} -{\beta_i}^2 -\gamma_i \right) - 
\left( \delta_i - \alpha_i \beta_i \right)^2 >0\, \label{eq:sub_det}
\end{align}
for $i=1,2$ and finally the admissibility inequality \eqref{eq:admissible subsolution} yields on the left interface:
\begin{align}
& \nu_-(\vr_- \varepsilon(\vr_-)- \vr_1 \varepsilon( \vr_1))+\nu_- 
\left(\vr_- \frac{\abs{\vv_-}^2}{2}- \vr_1 \frac{C_1}{2}\right)\nonumber\\
\leq & \left[(\vr_- \varepsilon(\vr_-)+ p(\vr_-)) v_{-2}- 
( \vr_1 \varepsilon( \vr_1)+ p(\vr_1)) \beta_1 \right] 
+ \left( \vr_- v_{-2} \frac{\abs{\vv_-}^2}{2}- \vr_1 \beta_1 \frac{C_1}{2}\right)\, ;\label{eq:E_left}
\end{align}
on the middle interface:
\begin{align}
& \nu_0(\vr_1 \varepsilon(\vr_1)- \vr_2 \varepsilon( \vr_2))+\nu_0 
\left(\vr_1 \frac{C_1}{2}- \vr_2 \frac{C_2}{2}\right)\nonumber\\
\leq & \left[(\vr_1 \varepsilon(\vr_1)+ p(\vr_1)) \beta_1- 
( \vr_2 \varepsilon( \vr_2)+ p(\vr_2)) \beta_2 \right] 
+ \left( \vr_1 \beta_1 \frac{C_1}{2}- \vr_2 \beta_2 \frac{C_2}{2}\right)\, ;\label{eq:E_mid}
\end{align}
and on the right interface:
\begin{align}
&\nu_+(\vr_2 \varepsilon( \vr_2)- \vr_+ \varepsilon(\vr_+))+\nu_+ 
\left( \vr_2 \frac{C_2}{2}- \vr_+ \frac{\abs{\vv_+}^2}{2}\right)\nonumber\\
\leq &\left[ ( \vr_2 \varepsilon( \vr_2)+ p(\vr_2)) \beta_2- (\vr_+ \varepsilon(\vr_+)+ p(\vr_+)) v_{+2}\right] 
+ \left( \vr_2 \beta_2 \frac{C_2}{2}- \vr_+ v_{+2} \frac{\abs{\vv_+}^2}{2}\right)\, .\label{eq:E_right}
\end{align} 

Now we make the following choice of some unknowns. We set
\begin{align}
    \alpha_1 &= v_{-1}, \\
    \alpha_2 &= v_{+1}, \\
    \vr_1 &= \vr_2, \\
    \beta_1 &= \beta_2 =: \beta.
\end{align}
This way \eqref{eq:cont_mid} is satisfied trivially and \eqref{eq:mom_2_mid} simplifies to
\begin{equation}
    \gamma_1 - \frac{C_1}{2} = \gamma_2 - \frac{C_2}{2}.
\end{equation}
Moreover we can easily calculate that $\delta_i = \alpha_i\beta$ from \eqref{eq:mom_1_left} and \eqref{eq:mom_1_right}, which in turn yields $\nu_0 = \beta$ from \eqref{eq:mom_1_mid}. The admissibility inequality \eqref{eq:E_mid} is satisfied trivially as an equality. We define $\ep_1 = \frac{C_1}{2}-\gamma_1 - \beta^2$, $\ep_2 = C_1 - v_{-1}^2 - \beta^2 - \ep_1$, $\ep_2' = C_2 - v_{+1}^2 - \beta^2 - \ep_1$ and we make the ansatz $\ep_2 = \ep_2'$.

Now we can follow the arguments in \cite[Lemma 4.3 - Lemma 4.4]{ChiKre14} and transform \eqref{eq:cont_left}-\eqref{eq:E_right} to 
\begin{align}
&\nu_- (\vr_- - \vr_1) \, =\,  \vr_- v_{-2} -\vr_1  \beta, \label{eq:cont_left ss}  \\
&\nu_- (\vr_- v_{-2}- \vr_1 \beta) \, = \,  
\vr_- v_{-2}^2 - \vr_1(\beta^2 + \ep_1) +p (\vr_-)-p (\vr_1) \, ;\label{eq:mom_2_left ss}
\end{align}
\begin{align}
&\nu_+ (\vr_1-\vr_+ ) \, =\,  \vr_1  \beta - \vr_+ v_{+2}, \label{eq:cont_right ss}\\
&\nu_+ (\vr_1 \beta- \vr_+ v_{+2}) \, = \, \vr_1 (\beta^2 + \ep_1) - \vr_+ v_{+2}^2 +p (\vr_1) -p (\vr_+) 
\, ;\label{eq:mom_2_right ss}
\end{align}
\begin{align}
& \ep_1 > 0, \label{eq:sub_1 ss}\\
& \ep_2 > 0\, ;\label{eq:sub_2 ss}
\end{align}
\begin{align}
&(\beta-v_{-2})\left(p(\vr_-)+p(\vr_1)-2\vr_-\vr_1\frac{\ep(\vr_-)-\ep(\vr_1)}{\vr_--\vr_1}\right) \nonumber\\
\leq &\ep_1\vr_1(v_{-2}+\beta) - (\ep_1+\ep_2)\frac{\vr_-\vr_1(\beta-v_{-2})}{\vr_--\vr_1}\, ;\label{eq:E_left ss}
\end{align}
\begin{align}
&(v_{+2}-\beta)\left(p(\vr_1)+p(\vr_+)-2\vr_1\vr_+\frac{\ep(\vr_1)-\ep(\vr_+)}{\vr_1-\vr_+}\right) \nonumber\\
\leq &-\ep_1\vr_1(v_{+2}+\beta) + (\ep_1+\ep_2)\frac{\vr_1\vr_+(v_{+2}-\beta)}{\vr_1-\vr_+}\, .\label{eq:E_right ss}
\end{align} 

A key feature of the Chaplygin gas pressure law  \eqref{eq:chap} is that 
\begin{equation}
    P(r,s) := p(r)+p(s)-2rs\frac{\ep(r)-\ep(s)}{r-s} = 0
\end{equation}
for all $r \neq s$, $r,s > 0$. This is an important difference from the more common pressure law $p(\vr) = \vr^\gamma$, $\gamma \geq 1$, where $P(r,s) > 0$.

Therefore \eqref{eq:E_left ss}, \eqref{eq:E_right ss} simplify further to 
\begin{align}
    &0 \leq \ep_1\vr_1(v_{-2}+\beta) - (\ep_1+\ep_2)\frac{\vr_-\vr_1(\beta-v_{-2})}{\vr_--\vr_1}\, ;\label{eq:E_left ss2} \\
    &0\leq -\ep_1\vr_1(v_{+2}+\beta) + (\ep_1+\ep_2)\frac{\vr_1\vr_+(v_{+2}-\beta)}{\vr_1-\vr_+}\, .\label{eq:E_right ss2}
\end{align}

Observing that relations \eqref{eq:cont_left ss}-\eqref{eq:mom_2_right ss} consist of four equations for five unknowns $\vr_1,\nu_\pm, \beta, \ep_1$ we start with expressing $\nu_\pm, \beta$ and $\ep_1$ as functions of $\vr_1$ which we treat as a parameter. To this end we introduce the notation
\begin{align}
    R &:= \vr_- - \vr_+ ,\\
    A &:= \vr_- v_{-2} - \vr_+v_{+2},  \\
    u &:= v_{-2} - v_{+2}, \\
    B &:= \vr_-\vr_+u^2 - \frac{R^2}{\vr_-\vr_+} \label{eq:Bdef}
\end{align}
and observe that \eqref{eq:deltacond2} ensures that we have $u > 0$ and $B > 0$. We make the ansatz $\vr_1 > \max\{\vr_-,\vr_+\}$ and assuming $R \neq 0$ this yields 
\begin{align}
 \nu_- = \frac{A}{R} - \frac{\sqrt{B}}{R}\sqrt{\frac{\vr_1-\vr_+}{\vr_1-\vr_-}}, \label{eq:num}\\
 \nu_+ = \frac{A}{R} - \frac{\sqrt{B}}{R}\sqrt{\frac{\vr_1-\vr_-}{\vr_1-\vr_+}} \label{eq:nup}
\end{align} 
and we see that $\nu_- < \nu_+$ is satisfied. We can express $\beta$ from \eqref{eq:cont_left ss} as follows
\begin{equation}\label{eq:beta}
 \beta = \frac{\vr_-v_{-2}}{\vr_1} - \frac{(\vr_--\vr_1)A}{R\vr_1} - \frac{\sqrt{B}}{R\vr_1}\sqrt{(\vr_1-\vr_-)(\vr_1-\vr_+)}
\end{equation}
and finally we get from \eqref{eq:mom_2_left ss}
\begin{equation}\label{eq:ep1}
 \ep_1 = \left(\frac{\sqrt{B}}{R}\sqrt{\frac{\vr_1-\vr_+}{\vr_1-\vr_-}} - \frac{\vr_+u}{R}\right)^2\frac{\vr_-(\vr_1-\vr_-)}{\vr_1^2} - \frac{\vr_1-\vr_-}{\vr_1^2\vr_-},
\end{equation} 
which is useful to study the case $R > 0$. Note that we could also express $\ep_1$ using \eqref{eq:mom_2_right ss} to get a different expression, which is useful for studying the case $R < 0$, namely
\begin{equation}\label{eq:ep10}
 \ep_1 = \left(\frac{\sqrt{B}}{R}\sqrt{\frac{\vr_1-\vr_-}{\vr_1-\vr_+}} - \frac{\vr_-u}{R}\right)^2\frac{\vr_+(\vr_1-\vr_+)}{\vr_1^2} - \frac{\vr_1-\vr_+}{\vr_1^2\vr_+}.
\end{equation} 

In the case $R = 0$ the expressions are much easier, namely
\begin{align}
 \nu_- &= \frac{v_{-2}+v_{+2}}{2} - \frac{\vr_-u}{2(\vr_1-\vr_-)}, \label{eq:num2}\\
 \nu_+ &= \frac{v_{-2}+v_{+2}}{2} + \frac{\vr_-u}{2(\vr_1-\vr_-)}, \label{eq:nup2}\\
 \beta &= \frac{v_{-2}+v_{+2}}{2}, \label{eq:beta2}\\
 \ep_1 &= \frac{\vr_-u^2}{4(\vr_1-\vr_-)} + \frac{1}{\vr_1^2} - \frac{1}{\vr_1\vr_-} \label{eq:ep12}.
\end{align}

\begin{lemma}\label{l:41}
Let \eqref{eq:deltacond2} be satisfied. Then $\ep_1 > 0$ for all $\vr_1 > \max\{\vr_-,\vr_+\}$.
\end{lemma}
\begin{proof}
We start with the case $R = 0$. Then $\ep_1 > 0$ yields the following quadratic inequality for $\vr_1$
\begin{equation}
    (\vr_-^2u^2 - 4)\vr_1^2 + 8\vr_-\vr_1 - 4\vr_-^2 > 0
\end{equation}
and we notice  that \eqref{eq:deltacond2} is, in this case, equivalent to $\vr_-^2u^2 - 4 > 0$.  The quadratic expression on the left--hand side has two roots
\begin{equation}
    \vr_{1,12} = 2\vr_-\frac{2 \pm \vr_- u}{4 - \vr_-^2 u^2},
\end{equation}
and both of them satisfy $\vr_{1,12} < \vr_-$. Therefore we get that $\ep_1 > 0$ for all $\vr_1 > \vr_-$.

Let us now handle the case $R > 0$. First observe that for $\vr_1 \sil \vr_-$ we have $\ep_1 \sil \frac{B\vr_-}{R} > 0$ and hence it is enough to show that 
\begin{equation}
    \widetilde{\ep_1} = \frac{\vr_1^2}{\vr_-(\vr_1-\vr_-)}\ep_1 = \left(\frac{\sqrt{B}}{R}\sqrt{\frac{\vr_1-\vr_+}{\vr_1-\vr_-}} - \frac{\vr_+u}{R}\right)^2 - \frac{1}{\vr_-^2} > 0
\end{equation}
for all $\vr_1 > \vr_-$. It is also easy to check that $\sqrt{B} > \vr_+ u$ and thus $\widetilde{\ep_1}(\vr_1)$ is a decreasing function. We examine its limit as $\vr_1 \sil \infty$. We have
\begin{equation}
    \lim_{\vr_1 \sil \infty} \widetilde{\ep_1}(\vr_1) = \frac{(\sqrt{B}-\vr_+ u)^2}{R^2} - \frac{1}{\vr_-^2} = \left(\frac{\sqrt{B}-\vr_+ u}{R} + \frac{1}{\vr_-}\right)\left(\frac{\sqrt{B}-\vr_+ u}{R} - \frac{1}{\vr_-}\right).
\end{equation}
The first term in the product on the right hand side is clearly positive, so it remains to prove non--negativity of the second term. Indeed, this is equivalent to
\begin{equation}\label{eq:blabla}
    \sqrt{B} \geq \frac{R}{\vr_-} + \vr_+ u.
\end{equation}
Plugging in \eqref{eq:Bdef} and making square of \eqref{eq:blabla} we get the following quadratic inequality in terms of $u$
\begin{equation}\label{eq:bubu}
    \vr_+ u^2 - 2\frac{\vr_+}{\vr_-} u - \frac{R(\vr_-+\vr_+)}{\vr_-^2\vr_+} \geq 0
\end{equation}
and we observe that the inequality \eqref{eq:bubu} holds for all $u \geq \frac{1}{\vr_+} + \frac{1}{\vr_-}$, i.e. in the case \eqref{eq:deltacond2}.

The case $R<0$ can be treated similarly using \eqref{eq:ep10} instead of \eqref{eq:ep1} to express $\ep_1$. We omit the calculation here.
\end{proof}

It remains to check the order of the speeds of the interfaces $\nu_- < \nu_0 = \beta < \nu_+$ and the admissibility conditions \eqref{eq:E_left ss2} and \eqref{eq:E_right ss2}. More precisely, if it is possible to find $\ep_2$ positive  such that both of these inequalities hold. These questions are closely connected, as is demonstrated by   observing that
\begin{align}
\beta - \nu_- &= \frac{\vr_-}{\vr_1}(v_{-2} - \nu_-), \label{eq:ep2 1}\\
 v_{-2} - \beta &= \frac{\rho_1-\rho_-}{\rho_1}\left(v_{-2}-\nu_-\right), \label{eq:ep2 2}\\
 \nu_+ - \beta &= \frac{\vr_+}{\vr_1}(\nu_+ - v_{+2}), \label{eq:ep2 3}\\
 \beta - v_{+2} &= \frac{\rho_1-\rho_+}{\rho_1}\left(\nu_+-v_{+2}\right) \label{eq:ep2 4}
\end{align} 
and the following lemma.
\begin{lemma}\label{l:42}
Let \eqref{eq:deltacond2} be satisfied. Then $v_{-2} - \nu_- > 0$ and $\nu_+ - v_{+2} > 0$ for all $\vr_1 > \max\{\vr_-,\vr_+\}$.
\end{lemma}
\begin{proof}
If $R = 0$ the claim follows directly from \eqref{eq:num2} and \eqref{eq:nup2}. Next, we assume $R>0$ and since 
\begin{equation}
    v_{-2} - \nu_- = \frac{\sqrt{B}}{R}\sqrt{\frac{\vr_1-\vr_+}{\vr_1-\vr_-}} - \frac{\vr_+u}{R},
\end{equation}
the claim $v_{-2} - \nu_- > 0$ follows thanks to $\sqrt{B} > \vr_+ u$.

Concerning the expression $\nu_+ - v_{+2}$ we calculate similarly that
\begin{equation}
    \nu_+ - v_{+2} = \frac{\vr_-u}{R} -  \frac{\sqrt{B}}{R}\sqrt{\frac{\vr_1-\vr_-}{\vr_1-\vr_+}} > \frac{\vr_- u - \sqrt{B}}{R} > 0.
\end{equation}
The case $R < 0$ is proved using the same arguments.
\end{proof}

This shows that we have correct order of the interface speeds $\nu_- < \nu_0 = \beta < \nu_+$. Moreover, knowing now that the expressions $v_{-2} - \beta$ and $\beta - v_{+2}$ have positive signs we obtain from \eqref{eq:E_left ss2} and \eqref{eq:E_right ss2} that
\begin{align}
    &\ep_2 \leq \ep_1\left(\frac{v_{-2} + \beta}{v_{-2} - \beta}\frac{\vr_1-\vr_-}{\vr_-} - 1\right), \label{eq:adf1}\\
    &\ep_2 \leq \ep_1\left(\frac{v_{+2} + \beta}{v_{+2} - \beta}\frac{\vr_1-\vr_+}{\vr_+} - 1\right). \label{eq:adf2}
\end{align}
We want to show that the expressions on the right hand sides of \eqref{eq:adf1} and \eqref{eq:adf2} can be made positive by choosing $\vr_1$ large enough. In particular it is easy to observe that in the case $\beta = 0$, the choice $\vr_1 > 2\max\{\vr_-,\vr_+\}$ ensures that the right--hand sides are positive and thus it is possible to find $\ep_2 > 0$ satisfying both \eqref{eq:adf1}, \eqref{eq:adf2}. We get a subsolution and therefore using Proposition \ref{p:subs} infinitely many bounded admissible weak solutions.

Finally, we claim that using the Galilean transformation, the assumption $\beta = 0$ can be made without loss of generality. We use the following argument. Starting with any couple $(\vr_-,\vv_-)$, $(\vr_+,\vv_+)$ characterizing the Riemann initial data \eqref{eq:Riem} we fix $\vr_1 > 2\max\{\vr_-,\vr_+\}$ and express $\beta$ from \eqref{eq:beta}. If $\beta \neq 0$ we use the Galilean transformation with constant velocity $(0,\beta)$ and study the case with Riemann initial data $(\vr_-,(v_{-1}, v_{-2}-\beta))$, $(\vr_+, (v_{+1}, v_{+2}-\beta))$. We obtain infinitely many bounded admissible weak solutions $(\vr,\vv)(t,x)$ starting from this data and using the transformation backwards we get that $(\vr,\vv+(0,\beta))(t,x-\beta t)$ are solutions with the original initial data \eqref{eq:Riem}.

In the case $\vr_- = \vr_+$ and $\beta = 0$, the inequalities \eqref{eq:adf1}, \eqref{eq:adf2} are the same. Therefore choosing $\vr_1 > 2\vr_-$ and $\ep_2 = \ep_1\frac{\vr_1-2\vr_-}{\vr_-}$ we get that both \eqref{eq:adf1}, \eqref{eq:adf2} are satisfied as equalities and thus solutions constructed from this subsolution satisfy the energy equality. This concludes the proof of Theorem \ref{t:main}.

\section{Proof of Theorem \ref{t:main2}}\label{s:Proof2}

Let us fix such initial data $\vr_-, \vr_+, \vv_{-}, \vv_{+}$  that 
\begin{equation}\label{eq:Th2cond2}
    \max\left\{\frac{1}{\vr_-}, \frac{1}{\vr_+}\right\} < v_{-2} - v_{+2} < \frac{1}{\vr_-} + \frac{1}{\vr_+}.
\end{equation}
Once again our goal is to find a single admissible fan subsolution. We follow the steps from the previous section and first simplify the set of equations and inequalities to \eqref{eq:cont_left ss}-\eqref{eq:sub_2 ss}, \eqref{eq:E_left ss2}, \eqref{eq:E_right ss2}. Continuing further, we get again that $B > 0$ and search for subsolutions with $\vr_1 > \max\{\vr_-,\vr_+\}$. The expressions \eqref{eq:num}-\eqref{eq:ep12} hold here as well, however instead of Lemma \ref{l:41} we show the following.
\begin{lemma}\label{l:51}
Let \eqref{eq:Th2cond2} be satisfied. There exists a unique $\vr_{max}$ such that $\ep_1 > 0$ for all $\vr_1 \in (\max\{\vr_-,\vr_+\}, \vr_{max})$ and $\ep_1 < 0$ for all $\vr_1 > \vr_{max}$.
\end{lemma}
\begin{proof}
We start with the case $R = 0$. Note that from \eqref{eq:Th2cond2} we have $0 < u < \frac{2}{\vr_-}$. Using \eqref{eq:ep12} we easily obtain that 
\begin{equation}
\widetilde{\ep_1}(\vr_1) = \frac{\vr_1^2}{\vr_-(\vr_1-\vr_-)}\ep_1(\vr_1) = \frac{\vr_1^2 u^2}{4(\vr_1-\vr_-)^2} - \frac 1{\vr_-^2}
\end{equation}
 is a decreasing function with a single root $\vr_{max} = \frac{2\vr_-}{2-\vr_-u} > \vr_-$, which proves the claim.
 
Next we handle the case $R > 0$. Once again we define the function 
\begin{equation}
    \widetilde{\ep_1}(\vr_1) = \frac{\vr_1^2}{\vr_-(\vr_1-\vr_-)}\ep_1(\vr_1) = \left(\frac{\sqrt{B}}{R}\sqrt{\frac{\vr_1-\vr_+}{\vr_1-\vr_-}} - \frac{\vr_+u}{R}\right)^2 - \frac{1}{\vr_-^2}
\end{equation}
and observe that under the condition \eqref{eq:Th2cond2} it holds $\sqrt{B} > \vr_+ u$. Indeed, 
we assume $u^2 > \frac{1}{\vr_+^2}$ that is more strict than $u^2 > \frac{1}{\vr_+}(\frac{1}{\vr_+}-\frac{1}{\vr_-})$.
Hence the function $\widetilde{\ep_1}$ is once again obviously decreasing and it is a matter of a straightforward calculation to find its root, which is 
\begin{equation}
    \vr_{max} = \frac{\vr_+ B - \vr_-\left(\frac{R}{\vr_-} + \vr_+ u\right)^2}{B - \left(\frac{R}{\vr_-} + \vr_+ u\right)^2} = \frac{2\vr_-\vr_+ u + 2R}{2\vr_+ u + \frac{R(\vr_-+\vr_+)}{\vr_-\vr_+} - \vr_-\vr_+ u^2}.
\end{equation}
One can check that the denominator of the last expression is positive for $u > \frac{1}{\vr_+}-\frac{1}{\vr_-}$ and under the same condition it also holds $\vr_{max} > \vr_-$.

The case $R < 0$ can be done in a similar way to $R>0$ using expression \eqref{eq:ep10}. We omit the details.
\end{proof}

The proof of Lemma \ref{l:51} provides the necessary arguments for the analog of Lemma \ref{l:42} to also hold under the assumption \eqref{eq:Th2cond2}. More precisely, we can show the following.
\begin{lemma}\label{l:52}
Let \eqref{eq:Th2cond2} be satisfied. Then $v_{-2} - \nu_- > 0$ and $\nu_+ - v_{+2} > 0$ for all $\vr_1 \in (\max\{\vr_-,\vr_+\},\vr_{max})$.
\end{lemma}

Therefore we know the correct order of the interface speeds $\nu_- < \nu_0 = \beta < \nu_+$. Since we also know that  $v_{-2} - \beta$ and $\beta - v_{+2}$ are positive, we can transform the admissibility inequalities \eqref{eq:E_left ss2}, \eqref{eq:E_right ss2} once again into \eqref{eq:adf1}, \eqref{eq:adf2}. We observe that in the case $\beta = 0$, the inequalities \eqref{eq:adf1},\eqref{eq:adf2} can be satisfied with a positive $\ep_2$ under the condition $\vr_1 > 2\max\{\vr_-,\vr_+\}$. Thus, in order to conclude the proof, it only remains to ensure that $\vr_{max} > 2\max\{\vr_-,\vr_+\}$.

In the case $R > 0$ this condition yields the following quadratic inequality for $u$:
\begin{equation}
    \vr_-\vr_+^2 u^2 - \vr_+^2 u - R = (\vr_+ u - 1)(\vr_-\vr_+ u + R)> 0,
\end{equation}
and hence we end up with $\vr_{max} > 2\vr_-$ whenever $u > \frac{1}{\vr_+}$. The cases $R = 0$ and $R < 0$ are similar.

Finally, as in the previous section, we use the the Galilean transformation argument  to claim that the choice $\beta = 0$ can be made without loss of generality.

In the case $R = 0$, i.e. $\vr_- = \vr_+$, and $\beta = 0$, the inequalities \eqref{eq:adf1}, \eqref{eq:adf2} are the same, so choosing $\ep_2 = \ep_1\frac{\vr_1 - 2\vr_-}{\vr_-} > 0$ we construct solutions satisfying the energy equality.


\section{Maximal dissipation criteria} \label{s:Rem}

Theorems \ref{t:main} and \ref{t:main2} are another examples in now a quite large series of results proving that the energy inequality \eqref{eq:energy} in itself is not strong enough to single out a unique physical solution in a set of (bounded) weak solutions in more than one space dimension. The list of such results was provided in the Introduction.

A natural question therefore rises, whether there is  another criterion strong enough to restore the uniqueness in the  multi--dimensional case. One of the concepts mentioned in the literature is the {\it principle of maximal dissipation of energy}. Roughly speaking, physical solutions should be those solutions that dissipate the most energy, or in mathematical terms produce the most entropy.

Motivated by Dafermos \cite{Da1} the {\em entropy rate} admissibility criterion was studied in \cite{ChiKre14} and \cite{Fe}. In order to define this criterion, we introduce the total energy $E_L[\vr,\vv]$ of a solution $(\vr,\vv)$ to \eqref{eq:euler} and the energy dissipation rate $D_L[\vr,\vv]$ as follows
\begin{align}
&E_L[\vr,\vv](t) = \int_{(-L,L)^2} \left(\vr\ep(\vr) + \vr\frac{\abs{\vv}^2}{2}\right)\dx, \label{eq:energy L}\\
&D_L[\vr,\vv](t) = \frac{\mathrm{d}_+ E_L[\vr,\vv](t)}{\dt}. \label{eq:dissipation rate L}
\end{align} 
Note that the restriction to the finite set $(-L,L)^2$ is necessary, since the solutions discussed in this paper clearly have the property $E_L[\vr,\vv] \sil \infty$ as $L \sil \infty$. Also note that $D_L[\vr,\vv]$ is a non--positive quantity for any admissible weak solution $(\vr,\vv)$.

\begin{definition}[Entropy rate admissible solution]\label{d:entropy rate}
A weak solution $(\vr,\vv)$ of \eqref{eq:euler} is called \textit{entropy rate admissible} if 
there exists $L^* > 0$ such that there is no other weak solution $(\overline{\vr},\overline{\vv})$ with the property that for some $\tau\geq 0$, $(\overline{\vr},\overline{\vv})(x,t)= (\vr,\vv)(x,t)$ on $\R^2 \times [0, \tau]$
and $ D_L[\overline{\vr},\overline{\vv}](\tau) < D_L[\vr,\vv](\tau) $ for all $L \geq L^*$.
\end{definition} 

Feireisl in \cite{Fe} proved global existence of infinitely many admissible weak solutions to \eqref{eq:euler} with the pressure $p(\vr) = \vr^\gamma$, $\gamma > 1$ starting from smooth density $\vr^0$ and some irregular velocity $\vv^0$. He also proved that none of these solutions is entropy rate admissible.

On the other hand, Chiodaroli and Kreml \cite{ChiKre14} showed that for the pressure $p(\vr) = \vr^\gamma, \gamma \in [1,3)$, there exist Riemann initial data \eqref{eq:Riem} for which the classical 1D solution consisting of two admissible shocks is not entropy rate admissible, in particular there exist solutions constructed by the method of De Lellis and Sz\' ekelyhidi \cite{DLSz1}, \cite{DLSz2} that dissipate more total energy than the classical 1D self-similar solution. Note that the authors in \cite{ChiKre14} do not claim that any of the nonstandard solutions is entropy rate admissible, on the contrary, as was also shown in \cite{Fe}, the feature of the construction of solutions based on the method of De Lellis and Sz\' ekelyhidi is that none of these solutions can be entropy rate admissible, because for each such solution there exists another one constructed by the same method that dissipates more energy.

As was pointed out by E. Feireisl in private discussions on the topic of maximal dissipation, the entropy rate admissibility criterion as defined in Definition \ref{d:entropy rate} only measures the dissipation of the total energy. A stronger version of the maximal dissipation criterion would take into account that the maximality should be achieved locally everywhere in space.

In order to accomodate this idea and to propose another admissibility criterion based on the notion of maximal dissipation, we first introduce the energy dissipation measure (or entropy production measure) $\mu[\vr,\vv] \in \mathcal{M}((0,\infty)\times\R^2)$ for a weak solution $(\vr,\vv)$ as follows
\begin{equation}
       \pat\left(\vr\ep(\vr) + \vr \frac{\abs{\vv}^2}{2}\right) + \Div\left(\left(\vr\ep(\vr) + \vr \frac{\abs{\vv}^2}{2} + p(\vr)\right)\vv\right) =: - \mu[\vr,\vv].
\end{equation}
It is clear that for any admissible weak solution $\mu[\vr,\vv]$ is a non--negative measure. In the terminology of hyperbolic conservation laws this measure is called {\it entropy production measure}.

\begin{definition}[Entropy production measure admissible solution]\label{d:entropy measure}
A weak solution $(\vr,\vv)$ of \eqref{eq:euler} is called \textit{entropy production measure admissible} if there is no other weak solution $(\overline{\vr},\overline{\vv})$ with the property that for some $\tau\geq 0$, $(\overline{\vr},\overline{\vv})(x,t)= (\vr,\vv)(x,t)$ on $\R^2 \times [0, \tau]$
and $ \mu[\overline{\vr},\overline{\vv}] > \mu[\vr,\vv]$ on $(\tau,\infty)\times \R^2$.
\end{definition}

We point out that nonstandard solutions to the system \eqref{eq:euler} with the pressure law $p(\vr) = \vr^\gamma$, $\gamma \geq 1$ constructed in \cite{ChiKre14} have the energy dissipation measure incomparable with the energy dissipation measure of the classical 1D solution consisting of two shocks, starting from the same initial data. This is a direct consequence of the fact, that the energy dissipation measure is supported on the shocks (for the classical 1D solution) or interfaces (in the case of nonstandard solutions) and in the construction in \cite{ChiKre14} the speeds of the interfaces cannot coincide with the shock speeds. The same also holds  for solutions constructed in \cite{ChiKre17} and \cite{MarKli17}.

Therefore it was possible to hope that the classical 1D solutions may be entropy production measure admissible,  at least in the class of solutions which can be constructed using the theory of De Lellis and Sz\' ekelyhidi, i.e. that one cannot construct nonstandard solutions with dissipation measure strictly larger than the dissipation measure of the classical 1D solution. However, solutions constructed in Theorem \ref{t:main2} violate this claim.

\begin{theorem}
Let $p(\vr) = -\vr^{-1}$. Assume that the Riemann initial data \eqref{eq:Riem} satisfy
\begin{equation}\label{eq:Th2cond3}
    \max\left\{\frac{1}{\vr_-}, \frac{1}{\vr_+}\right\} < v_{-2} - v_{+2} < \frac{1}{\vr_-} + \frac{1}{\vr_+}.
\end{equation}
Then the classical 1D solution to \eqref{eq:euler}, \eqref{eq:energy}, \eqref{eq:Riem} consisting of two (for $v_{-1} = v_{+1}$) or three (for $v_{-1} \neq v_{+1}$) contact discontinuities is not entropy rate admissible and it is not entropy production measure admissible.
\end{theorem}
\begin{proof}
The theorem  is a direct consequence of Theorem \ref{t:main2}. Since the classical 1D solution $(\vr_c,\vv_c)$ consists only of contact discontinuities, we have that $D_L[\vr_c,\vv_c](t) = 0$ for all $L > 0$ and $t > 0$ and also $\mu[\vr_c, \vv_c] = 0$.

On the other hand, nonstandard solutions $(\vr,\vv)$ in Theorem \ref{t:main2} can be constructed in such a way that they satisfy the energy inequality \eqref{eq:energy} as a strict inequality on the interfaces $x = \nu_- t$, $x = \nu_+ t$. Thus these solutions satisfy $D_L[\vr,\vv](t) < 0$ for all $t > 0$ and $L > 0$ large enough with respect to $t$, and also $\mu[\vr, \vv] > 0$.
\end{proof}

We conclude with the remark, that
all results in this paper can be directly generalized to any space dimension higher than $2$.

\section*{Acknowledgements}
O. Kreml and V. M\'acha were supported by the GA\v CR (Czech Science Foundation) project GJ17-01694Y in the general framework of RVO: 67985840. O. Kreml acknowledges the support of the Neuron Impuls Junior project 18/2016.

\end{document}